\documentclass[a4paper,oneside,11pt]{amsart}
\usepackage[russian]{babel}

\usepackage{comment}
\usepackage[all,arc]{xy}
\usepackage[all]{xy}
\usepackage{braket}
\usepackage{latexsym,amssymb,amsmath}
\usepackage[hidelinks]{hyperref}

\usepackage{multirow}
\usepackage{hhline}

\newtheorem{lemma}{Lemma}
\newtheorem{proposition}{Proposition}
\newtheorem{theorem}{Theorem}

\newcommand{\OM}{\mathop{{\rm\Omega}}\nolimits}
\newcommand{\POM}{\mathop{{\rm P\Omega}}\nolimits}
\newcommand*{\GL}{\mathop{{\rm GL}}\nolimits}
\newcommand*{\SL}{\mathop{{\rm SL}}\nolimits}
\newcommand*{\GO}{\mathop{{\rm GO}}\nolimits}
\newcommand*{\SO}{\mathop{{\rm SO}}\nolimits}
\newcommand*{\CO}{\mathop{{\rm CO}}\nolimits}
\newcommand*{\CSO}{\mathop{{\rm CSO}}\nolimits}
\newcommand*{\SU}{\mathop{{\rm SU}}\nolimits}
\newcommand*{\GU}{\mathop{{\rm GU}}\nolimits}
\newcommand*{\Sp}{\mathop{{\rm Sp}}\nolimits}

\newcommand*{\PSL}{\mathop{{\rm PSL}}\nolimits}
\newcommand*{\PSU}{\mathop{{\rm PSU}}\nolimits}

\newcommand*{\PSp}{\mathop{{\rm PSp}}\nolimits}

\newcommand*{\Sz}{\mathop{{\rm Sz}}\nolimits}

\newcommand*{\ord}{\mathop{{\rm ord}}\nolimits}
\newcommand*{\F}{\mathop{\mathbb{F}}\nolimits}

\newcommand*{\bsm}{\left(\begin{smallmatrix}}
\newcommand*{\esm}{\end{smallmatrix}\right)}
\newcommand*{\bp}{\begin{pmatrix}}
\newcommand*{\ep}{\end{pmatrix}}
\newcommand*{\bpm}{\begin{pmatrix}}
\newcommand*{\epm}{\end{pmatrix}}
\newcommand*{\bbm}{\begin{bmatrix}}
\newcommand*{\ebm}{\end{bmatrix}}

\newcommand*{\ov}{\overline}

\newcommand*{\al}{\alpha}

\newcommand*{\eps}{\varepsilon}
\newcommand*{\ga}{\gamma}

\newcommand*{\lam}{\lambda}

\newcommand*{\Om}{\Omega}
\renewcommand*{\phi}{\varphi}

\addto\captionsrussian{}
\addto\captionsrussian{}
\addto\captionsrussian{}

\begin{document}
\footskip=30pt
\date{}

\title{Simple groups with Brauer trees of principal blocks \\ in the shape of a star}
\author{Andrei Kukharev}

\address{Siberian Federal University, Krasnoyarsk, Russia} 
\email{kukharev.av@mail.ru}

\renewcommand{\baselinestretch}{1.2}

\keywords{Brauer tree, Brauer star, simple group, finite group}.

\begin{abstract}
We have found a list of finite simple groups with cyclic Sylow $p$-subgroup whose principal $p$-blocks have Brauer trees in the shape of a star, that is a tree of diameter at most $2$. Moreover, for an arbitrary finite group $G$ with cyclic Sylow $p$-subgroup, we have obtained a necessary condition when the Brauer tree of the principal $p$-block of $G$ is a star.
\end{abstract}

\maketitle
\pagestyle{plain}


\section*{Introduction}

Let $G$ be a finite group, $p$ be a prime number dividing the order of $G$. Suppose that a Sylow $p$-subgroup of $G$ is cyclic. Then the Brauer graph of a $p$-block of $G$ is uniquely defined. Moreover, this graph is a tree.

Denote by $\mathfrak{X}_p$ the class of finite groups with non-trivial cyclic Sylow $p$-subgroup such that the Brauer tree of the principal $p$-block is a \emph{star}, that is a tree of diameter at most $2$. 
Note that the Brauer graph of a $p$-block with cyclic defect group is a star if and only if every $p$-modular irreducible character of this block lifts to an ordinary irreducible character (see \cite[Lemma 3.1]{WilZal14}).

We are interested in the description of the class $\mathfrak{X}_p$. The problem arises from the work \cite{Blau}, where the author studied the properties of such groups. For instance, he has shown that if $G \in \mathfrak{X}_p$ and $G$ is not a $p$-solvable group, then the star has an even number of edges. A similar problem was studied in \cite{His87}, where the author considered the class $\mathfrak{L}_p$ of groups all whose absolutely irreducible $p$-modular characters are liftable. 

It is known that if $G$ is a $p$-solvable group with cyclic Sylow $p$-subgroup, then Brauer trees of all $p$-blocks are starts, in particularly $G \in \mathfrak{X}_p$. But there exist also nonsolvable groups having this property. For instance, $A_5 \in \mathfrak{X}_3$.

The main goal of this work is to find all simple finite groups with the property $\mathfrak{X}_p$. But some non-simple groups (namely, symmetric and classical groups) also will be covered during our study. 

For simple groups, we have obtained the following result.

\begin{theorem}\label{t-main}
Let $G$ be a finite simple group, and let $p$ be a primer dividing the order of $G$. Then $G \in \mathfrak{X}_p$ if and only if one of the following statements holds.

1) $G= C_p$;

2) $G= \PSL_2(q)$, $p\neq 2$ and $p$ divides $q \pm 1$;

3) $G= \PSL_3(q)$, $p\neq 2$ and $p$ divides $q + 1$;

4) $G= \PSU_3(q^2)$, $p\neq 2$ and $p$ divides $q - 1$;

5) $G= A_5$, $p \in \{3, 5\}$;

6) $G= A_6$ and $p = 5$;

4) $G= \Sz(q)$, $q= 2^{2n+1}$ ($n\geq 1$), where $p \neq 2$ divides $q-1$ or $q+r+1$;

5) $G= {}^2G_2(q^2)$, $q^2= 3^{2n+1}$ ($n\geq 1$), where $p \neq 2$ divides $q^2-1$ or $q^2+\sqrt{3} q+1$;

6) $G \in \{ M_{11}, M_{12}, J_3 \}$ and $p=5$;

7) $G= J_1$ and $p \in \{3, 5\}$.
\end{theorem}

Moreover, for an arbitrary finite group $G$, the following theorem gives a necessary condition when $G \in \mathfrak{X}_p$.

\begin{theorem}\label{t-ext}
Let $G$ be a non-$p$-solvable group with a non-trivial cyclic Sylow $p$-subgroup $P$. Then there exists the smallest normal subgroup $K$ in $G$ properly containing $O_ {p'}(G)$, and the quotient group $L = K / O_{p'}(G) $ is simple non-abelian. Moreover, if $G \in \mathfrak{X}_p$, then $L \in \mathfrak{X}_p$.
\end{theorem}

\section{Preliminaries}\label{S-simple}

Recall some basic facts about Brauer trees of finite groups. We refer the reader to  \cite{Benson, Feit} for details.

Let $p$ be a primer dividing the order of a finite group $G$. We will denote by $\mathrm{Irr}(G)$ the set of irreducible ordinary characters of $G$ and by $\mathrm{IBr}_p(G)$ the set of irreducible Brauer ($p$-modular) characters of $G$. The restriction $\chi^o$ of $\chi \in \mathrm{Irr}(G)$ to the set on $p$-regular elements of $G$ can be decomposed as 
$$
\chi^o = \sum\limits_{\phi \in \mathrm{IBr}_p(G)} d_{\chi\phi} \phi.
$$
The coefficients $d_{\chi\phi}$ are called the \emph{decomposition} numbers. They form a decomposition matrix.

The \emph{Brauer graph} of $G$ is an undirected graph, the vertices of which are labelled by elements of $\mathrm{Irr}(G)$, and the edges are labelled by elements of $\mathrm{IBr}(G)$.
Some vertices may be labelled by a few irreductible ordinary characters $\chi_1,..., \chi_m$ if they have the same restriction to the $p$-regular conjugacy classes. We call a such vertice \emph{exceptional} with multiplicity $m$. Two vertices labeled by $\chi, \psi \in \mathrm{Irr}(G)$ are adjacent if there exists $\phi \in \mathrm{IBr}_p(G)$ such that $d_{\chi \phi} \neq 0$ and $d_{\psi \phi} \neq 0$. The connected components of the Brauer graph are called \emph{$p$-blocks} of $G$.

Let $B$ be a $p$-block of $G$ with cyclic defect group. 
Then the Brauer graph corresponding to $B$ is a tree, which we will denote by $\tau(B)$. If $e$ is the number of edges in $\tau(B)$, then $\tau(B)$ has $e+1$ vertices. A $p$-block may have not more then one exceptional vertice (with multiplicity $m>1$).

A vertex of a Brauer graph is called \emph{non-real} if its character has a non-real value for some $p$-regual element. The \emph{real stem} of a Brauer tree is a subtree obtained by removing all non-real vertices. The real stem always has the shape of a straight line (see \cite[p.~3]{H-L}).

\begin{lemma}\cite[p.\,212]{Benson}\label{e}
If a group $G$ has a cyclic Sylow $p$-subgroup $P$, then the number $e$ of edges of Brauer tree of the principal $p$-block of $G$ is equal to $|N_G(P) / C_G(P)|$, and the multiplicity of the exceptional vertice is $m = (|P|-1)/e$.
\end{lemma}

\begin{lemma}\cite[Theorem 1, Corollary 1]{Blau}\label{tree-prop}
Let $G$ be a simple non-abelian group with a non-trivial cyclic Sylow $p$-subgroup $P$. Suppose that the Brauer tree of the principal $p$-block of $G$ is a star with $e$ edges. Then
1) $e$ is even, 
2) if the number $|C_G(P)|$ is odd, then all involutions of $G$ form a unique conjugacy class.
\end{lemma}

\begin{lemma}\label{e_subgroup}
Suppose that $G$ is a group with cyclic Sylow $p$-subgroup $P$. Suppose that $H$ is a normal subgroup of $G$ such that $(|G/H|,p)=1$. Let $e_G, e_H$ be the numbers of edges of Brauer tree of the principal $p$-block of the groups $G$ and $H$, correspondently. Then $e_H \mid e_G$.
\end{lemma}
\begin{proof}
Since $P \in Syl_p(H)$, applying Frattini's argument, we obtain
$$e_H = |N_H(P)| / |C_H(P)| = |C_G(P) H|/|G| \cdot e_G,$$
where $C_G(P) H$ is a subgroup of $G$, because $H$ is normal in $G$.
\end{proof}
Let $(\tau, Q)$ be the Brauer tree of a block $B$ with the exceptional vertex $Q$. Then the tree $(\tau,Q)^n$ is obtained by winding up $\tau$ around $Q$ which created $n$ branches, where the original tree is considered as one of these branches. 

We say that the Brauer trees of blocks $B_1$ and $B_2$ are \emph{similar} if there is a tree $(\tau, Q)$ such that $\tau(B_1)= (\tau,Q)^m$ and $\tau(B_2)= (\tau,Q)^n$ for some $m, n \in \mathbb{N}$.

We denote by $B_0(G)$ the principal $p$-block of $G$, by $\tau_0(G)$ the Brauer tree of $B_0(G)$, and by $e_0(G)$ or $e_G$ the number of edges in $\tau_0(G)$.

\begin{lemma}\label{L-tree}
Let $G$ be a group with a cyclic Sylow $p$-subgroup $P$. Let $H$ be a normal subgroup of $G$ of index coprime to $p$. Then:

1) $\tau_0(G)$ is similar to $\tau_0(H)$;

2) If $\tau_0(G)$ is a line, then the same holds true for $\tau_0(H)$.
\end{lemma}
\begin{proof}
The first part of the lemma follows from \cite[Lemma 4.2]{Feit}.

Suppose that  $\tau_0(G)$ is a line with $e_G$ edges. Since $\tau_0(G)$ is similar to $\tau_0(H)$, we have that $\tau_0(G)= \tau^m$ and $\tau_0(H)= \tau^n$ for some tree $\tau$.
If $e_G$ is odd, then $\tau$ coincides with $\tau_0(G)$. It follows from $e_H\leq e_G$ that $\tau_0(H)=\tau$.

If $e_G$ is even, then either $\tau_0(G)=\tau$ or $\tau_0(G)=\tau^2$. In both cases, we obtain that $\tau_0(H)$ is a line.

\end{proof}

The following Lemma describes some simple properties of the class $\mathfrak{X}_p$.

\begin{proposition}\label{some_prop}
Let $G$ be a finite group, and $P$ be a Sylow $p$-subgroup of $G$.

1) If $G$ is $p$-solvable and $P$ is cyclic, then $G \in \mathfrak{X}_p$.

2) $G \in \mathfrak{X}_p$ if and only if $G / O_{p'}(G) \in \mathfrak{X}_p$.

3) Suppose that $H$ is a normal subgroup of $G$ such that $|G/H|$ is coprime to $p$. If $G \in \mathfrak{X}_p$, then $H \in \mathfrak{X}_p$.
\end{proposition}
\begin{proof}
1) If $G$ is a $p$-solvable group, then according to \cite[Lemma X.4.2]{Feit}, the Brauer tree of any $p$-block of $G$ with a cyclic defect group is a star.

2) The second statement holds because the kernel of the principal $p$-block of $G$ is equal to $O_{p'}(G)$.

3) The third statement follows from Lemmas \ref{L-tree} and \ref{e_subgroup}. 
\label{L-tree}

\end{proof}

\begin{proof}[Proof of Theorem 2]

Let $G$ be a non-$p$-solvable group with a non-trivial cyclic Sylow $p$-subgroup $P$. Let $H = G / O_{p'}(G)$.
Then, by \cite[Lemma 6.1]{Naehrig}, there is a unique minimal normal subgroup $L$ in $H$, and $L$ is simple.

Denote by $K$ a subgroup of $G$ such that $L \cong K/O_{p'}(G)$. Then $K$ is normal in $G$, and $K$ containing $O_ {p'}(G)$ properly, because $L \neq 1$.

Since $G$ is not $p$-solvable, the subgroup $H$ can not be a Frobenius group with kernel $P$. Thus, according to \cite[Lemma~5.1]{Blau}, $L$ contains $P$, hence $p$ doesn't divide $|G/K|$. It gives that $K$ is not abelian.

If $G \in \mathfrak{X}_p$, then using Proposition \ref{some_prop} we obtain that $K \in \mathfrak{X}_p$ and $L \in \mathfrak{X}_p$.
\end{proof}

The proof of Theorem \ref{t-main} is based on the classification of simple finite groups \cite{atlas}. We will consequentially consider cyclic groups, alternating groups, classical groups, exceptional groups of Lie type and sporadic groups.

\section{Cyclic, symmetric and alternating groups}

For cyclic groups, the result is simple. 
\begin{proposition}
$C_p \in \mathfrak{X}_p$ for any prime $p$.
\end{proposition}
\begin{proof}
The statement holds because each abelian group is solvable, i.e. it is $p$-solvable for any $p$ dividing the order of the group.
\end{proof}

Note that by definition the trivial group is not in $\mathfrak{X}_p$.

For symmetric and alternating groups, we have the following result.

\begin{proposition}

1) $S_n \in \mathfrak{X}_p$ if and only if $p=2$ and $n \in \{2,3\}$, or $p=3$ and $n \in \{3,4,5\}$.

2) $A_n \in \mathfrak{X}_p$ if and only if $p=3$ and $n \in \{3,4,5\}$, or $p=5$ and $n \in \{5,6\}$.
\end{proposition}
\begin{proof}
1) Let $G = S_n$ ($n \geq 2$). A sylow $p$-subgroup $P$ of $G$ is cyclic if and only if $p \leq n < 2p$ (or $n/2 < p \leq n$). Suppose that $P$ is cyclic. Since all ordinary irreducible characters of $S_n$ are real, they lay on the real steam which is a line, and there are no exceptional characters (i.e. multiplicity $m=1$). Thus, the number of edges in $\tau_0(G)$ is $e = |P| - 1$. This tree is a star if and only if $e \leq 2$. This holds if and only if $|P|=3$ and $n \in \{ 3,4,5 \}$, or $|P|=2$ and $n \in \{ 2,3 \}$.

2) For $G = A_n$ ($n \geq 3$), the number of edges in $\tau_0(G)$ is $e=p-1$ if $n/2 < p < n-1$, and $e=(p-1)/2$ if $p \in \{n-1, n\}$ (see \cite[p.282]{His87}). It gives desired.

\end{proof}

\section{Sporadic groups}

\begin{proposition}
Let $G$ be one of the sporadic groups. Then $G \in \mathfrak{X}_p$ if and only if one of the following statements holds.

1) $G=M_{11}$ and $p=5$;

2) $G=M_{23}$ and $p=5$;

3) $G=J_1$ and $p \in \{3,5\}$;

4) $G=J_3$ and $p = 5$.
\end{proposition}
\begin{proof}
Brauer trees of sporadic groups for most values of $p$ can be found in \cite{H-L}. Some trees can be also easily constructed from decomposition matrices \cite{Bre}. Also, using information about the orders of the centralizer $C_G(P)$ and normalizer $N_G(P)$ of a Sylow $p$-subgroup $P$ of $G$, it is easy to show with Lemma \ref{tree-prop} that $\tau_0(G)$ is not a star for most sporadic groups.

For instance, consider the Mathieu group $G = M_{23}$ of the order $2^7 \cdot 3^2 \cdot 5 \cdot 7 \cdot 11 \cdot 23$. For $p \in \{2,3\}$, a Sylow $p$-subgroup $P$ of $G$ is not cyclic. The Brauer tree of the principal $5$-block is a star, according to \cite{H-L}. For $p \in  \{7,23\}$, the order $|N_G(P)/C_G(P)|$ is odd, hence $G \notin \mathfrak{X}_p$. For $p=11$, the order of $C_G(P)$ is odd and $G$ has more than one conjugacy class of involutions, and  $G \notin \mathfrak{X}_p$ by Lemma \ref{tree-prop}. Other groups are analysed the same way. Only two groups, Baby Monster $BM$ for $p \in \{13,19\}$ and Fisher group $Fi_{24}'$ for $p \in \{11,13\}$, are required special consideration. In \cite[Lemma 5.1]{K-P-sporadic}, it has been showed that in these cases the graph $\tau_0(G)$ contains a path of length at least $3$.
\end{proof}

\section{Classical groups}

In this section, we will consider Brauer trees of finite classical groups: linear, symplectic, unitary, and orthogonal.

\begin{lemma}\label{2-sylow}
Let $G$ be a finite group which is not solvable. Then a Sylow $2$-subgroup of $G$ is not cyclic.
\end{lemma}
\begin{proof}
Let $G$ is a non-solvable group. Then by the Feit-Thompson theorem, $2$ divides $|G|$.

Suppose that a Sylow $2$-subgroup $P$ of $G$ is cyclic. Then according to \cite[Satz IV.2.8]{Huppert}, $G$ is $2$-nilpotent, i.e. $G$ has a normal subgroup $H$ such that $G/H \cong P$. It gives a contradiction.

\end{proof}

\begin{lemma}\label{p_div_q}\cite[Proposition 5.1]{S-W}
Let $G$ be a simple finite group of Lie type over $\mathbb{F}_q$, which is not isomorphic $\mathrm{PSL}_2(q)$. If $p \mid q$, then Sylow $p$-subgroup of $G$ is not abelian.
\end{lemma}

Thus, it is sufficient to consider only cases when $p \neq 2$, and when $p$ does not divide $q$ (except linear groups). We will use Stather's result \cite[Table 1]{Stat} about sizes of Sylow $p$-subgroups for such case. Also, the following Lemma shows that, in this case, the number of edges in a $p$-block depends only on  $\ord_p(q)$, the multiplicative order of $q$ modulo $p$.

\begin{lemma}\cite{Geck92}\label{independ_q}
Let $G$ be a finite group of Lie type over a field of $q$ elements with cyclic (non-trivial) Sylow $p$-subgroup. Suppose that $p \nmid q$, and let $d = \ord_p(q)$. Then $p$-modular decomposition matrix directly depends only on $d$, and does not depend on the particular choice of $q$ and $p$.
\end{lemma}

\begin{lemma}\label{polyg}

Suppose that $G \in \{ \GL_n(q),~ \Sp_{2n}(q) \}$, where $q \geq 2$, or  $G \in \{ \GU_n(q^2), \SO_{2n+1}(q), \\ \CSO^{\pm}_{2n}(q)  \}$, where $q$ is odd, or $G \in \{ \SU_n(q^2),~ \GO^{\pm}_{2n}(q)\}$, where $q$ is even. Suppose also that a Sylow $p$-subgroup of $G$ is cyclic, $p \nmid q$ and $p\neq 2$. Then the Brauer tree of the principal block of $G$ is a line.
\end{lemma}
\begin{proof}
The result has been proven in \cite{FS80} for $\GL_n(q)$ with any $q$, and in \cite{FS90} for the groups  $\Sp_{2n}(q)$, $\GU_n(q^2)$, $\SO_{2n+1}(q)$, and $\CSO^{\pm}_{2n}(q)$, when $q$ is odd.

Assume that $G$ is one of the groups $\Sp_{2n}(q)$ or $\GO^{\pm}_{2n}(q)$, where $q$ is even. Then, by Gow \cite{Gow}, each element of the group $G$ is a product of two involutions. It follows that each element of $G$ is conjugated to its inverse. Therefore, each ordinary character of $G$ is real-valued. Thus, the Brauer tree of each block with cyclic defect group coincides with its real stem, i.e. it has the shape of a line.

Assume that $G = \SU_n(q^2)$, $q$ is even. Since $G$ is quasi-simple, it follows from \cite[Section 6]{Rob} that each non-exceptional character in $B_0(G)$ is rational-valued, and hence it is located on the real stem. Thus, $\tau_0(G)$ is a line.
\end{proof}

Thus, for these groups, we need just to find the number $e_0(G)$ of edges in the principal block of $G$, i.e. the order of $N_G(P)/C_G(P)$. As we will see soon, for general linear groups, this number is equal to the multiplicative order of $q$ modulo $p$. 
This is not always true for other classical groups, but the approach is similar.

\subsection{Linear groups}\label{linear}

Recall that the order of the general linear group $\mathrm{GL}(n,q)$ over the finite field $\mathbb{F}_q$ with $q$ elements is equal to
\begin{equation}
|\GL_n(q)| = q^{n(n-1)/2}\cdot (q-1)\cdot\ldots\cdot (q^n-1). 
\end{equation}

The special linear group $\mathrm{SL}(n,q)$ is a normal subgroup of $\mathrm{GL}(n,q)$ of index $q-1$. The projective special linear group $\mathrm{PSL}(n,q)$ is obtained from $\mathrm{SL}(n,q)$ by factoring out its center $Z$ whose order equals $(n,q-1)$.

\begin{lemma}\cite[Lemma~2]{K-P15b}\label{GL_sylow_cyclic}
Let $G=\GL_n(q)$, $p \nmid q$ and $d = \ord_p(q)$. Then

1) A Sylow $p$-subgroup $P$ of $G$ is cyclic and non-trivial if and only if $d \leq n < 2d$. 

2) If $d \leq n < 2d$, then $|N_G(P)/C_G(P)|=d$.
\end{lemma}

\begin{proposition}\label{GL-tree}
Let $G=\GL_n(q)$, $n \geq 2$. Then $G \in \mathfrak{X}_p$ if and only if one of the following statements holds.

1) $n=2$ and $p=q \in \{ 2,\,3\}$;

2) $n \in \{2,\,3\}$, $p \neq 2$ and $p \mid q+1$.

\end{proposition}
\begin{proof}
Assume that $p \mid q$. In this case, $P$ is cyclic if and only if $n=2$ and $q = p$. 
Let $G = \mathrm{GL}_2(p)$. Then $|N_G(P)|=p(p-1)^2$ and $|C_G(P)|=p(p-1)$. Hence, $\tau_0(G)$ is a line (by Lemma \ref{polyg}) with $|N_G(P)/C_G(P)|= p-1$ edges. Therefore, $\mathrm{GL}_2(p) \in \mathfrak{X}_p$ if and only if $p \in \{2,3\}$.

Assume now that $p \nmid q$, and $d = \ord_p(q)$.
Since $P$ is cyclic, by Lemma \ref{GL_sylow_cyclic} we have that $p \neq 2$, $n < 2d$ and $|N_G(P)/C_G(P)| = d$. In particulary, $d \geq 2$. Thus, $\tau_0(G)$ is a star if and only if $d=2$ and $n=2,3$.

\end{proof}

\begin{lemma}\label{GL-SL-tree}
Let $G=\GL_n(q)$, $H = \SL_n(q)$. 
Suppose that a Sylow $p$-subgroup $P$ of $H$ is cyclic (and non-trivial).

1) If $p \nmid q$, then $\tau_0(G)$ and $\tau_0(H)$ contain the same number of edges.

2) If  $p \nmid q$ and  $p \nmid q - 1$, then $\tau_0(G) = \tau_0(H)$.
\end{lemma}
\begin{proof}
1) According to Frattini's argument, $G = N_G(P)H$. Therefore,
$$G/N_G(P)=H/(H \cap N_G(P)) = H/N_H(P),$$
and $|N_G(P)|/|N_H(P)| = |G|/|H| = q-1$.

Since $p \nmid q$, the centralizers $C_G(P)$ and $C_H(P)$ coincide with Singer cycles in $G$ and $H$. They have orders $q^n-1$ and $(q^n-1)/(q-1)$, correspondingly. It gives that $|N_H(P)/C_H(P)| = |N_G(P)/C_G(P)|$, as desired.

2) If also $(|G/H|,p) = 1$, then by Lemma \ref{L-tree}, Brauer trees $\tau_0(G)$ and $\tau_0(H)$ are similar. And since they contain the same number of edges, they have the same shape.

\end{proof}

\begin{proposition}
Let $H = \PSL_n(q)$ and $G = \SL_n(q)$, where $n \geq 2$ and $q \geq 5$. Then $H \in \mathfrak{X}_p$ ($G \in \mathfrak{X}_p$) if and only if one of the following statements holds.

1) $n=2$ and $p = q = 5$;

2) $n=2$, $p \neq 2$ and $p \mid q-1$;

3) $n \in \{2,\,3\}$, $p \neq 2$ and $p \mid q+1$; 

\end{proposition}
\begin{proof}

We may assume that $p>2$, overwise Sylow $p$-subgroups of both $\SL_n(q)$ and $\PSL_n(q)$ for $q \geq 5$ are not cyclic.

The center $Z$ of $\SL_n(q)$ has the order $(n, q-1)$. Therefore, if $p \nmid q-1$, then $Z$ is in the kernel of the block $B_0(\SL_n(q))$; hence, this block consides with $B_0(\PSL_n(q))$.

First, assume that $n=2$. We will consider three cases: 1) $p \mid q$, 2) $p \mid q-1$, and 3) $p$ doesn't divide $q$ neither $q-1$.

1) Let $p \mid q$. Then a Sylow $p$-subgroup $P$ of $\SL_2(q)$ is cyclic if and only if $p=q$. According to \cite[Theorem 71.3]{Dornhoff}, the Brauer tree of the principal $p$-block of $\SL_2(p)$ is a line with $(p-1)/2$ edges. Thus, $\SL_2(p) \in \mathfrak{X}_p$ if and only if $p \leq 5$.

2) Let $p \mid q-1$. For $G=\SL_2(q)$, we have that $P \in Syl_p(G)$ is cyclic, the centralizer $C_G(P)= \langle \alpha \rangle$ and the normalizer $N_G(y)= \langle \alpha, \beta \rangle$, where $\alpha = \bsm \nu & 0 \\ 0 & \nu^{-1} \esm$ and $\beta = \bsm 0&1\\-1&0 \esm$. Here $\nu$ is a generator of the multiplicative group $\mathbb{F}_q^*$ (see \cite[p.\,230]{Dornhoff}). Therefore, $|N_G(P)/C_G(P)|=2$, and $G \in \mathfrak{X}_p$. 

Now consider $H=\PSL_2(q)$, where $p \mid q-1$. A Sylow $p$-subgroup of $H$ is also cyclic, and it follows from \cite{Burk76} that the tree $\tau_0(H)$ is a line with two edges.

3) Let $p \nmid q$ and $p \nmid q-1$. In this case, the tree $\tau_0(H)$ consides with $\tau_0(G)$. But  $\tau_0(G)$ consides also with $\tau_0(\GL_n(q))$ by Lemma \ref{GL-SL-tree}. Thus, the trees $\tau_0(H)$ and $\tau_0(G)$ are both stars if $p \mid q+1$ (see Proposition \ref{GL-tree}).

Now suppose that $n \geq 3$. In this case, a Sylow $p$-subgroup of $\SL_n(q)$ is cyclic if and only if $p \nmid q$ and $p \nmid q-1$. As well as above, the trees of all three groups $\PSL_n(q)$, $\SL_n(q)$ and $\GL_n(q)$ are the same, and we can apply Proposition \ref{GL-tree}.

\end{proof}

\subsection{Symplectic groups}\label{symp}

Let $V$ be a vector space of even dimension $n=2m$ over the field $\F_q$. Let $f: V\times V\to V$ be a skew-symmetric form on $V$, i.e. a bilinear form such that $f(u,v)= - f(v,u)$ for all $u, v\in V$. 
If $f$ is given by a matrix $W$, then $W= -W^t$, where $t$ means transpose.

A symplectic group $\Sp_{2m}(q)$ consists of invertible matrices $A$ of order $2m$ which preserve $f$,
i.e. $A W A^t= W$. It has the order
$$
|\Sp_{2m}(q)|= q^{m^2}\cdot (q^2-1)\cdot (q^4-1)\cdot\ldots \cdot (q^{2m}-1)\,.
$$

We may take \rule{0pt}{6mm} $W = \bsm 0 & I_m \\ -I_m & 0 \esm$. Then the rule
\begin{equation}\label{emb_sp}
\phi:~ A\mapsto \bsm A&0\\ 0& A^{-t}\esm 
\end{equation}
defines an embedding of $\GL_m(q)$ into $\Sp_{2m}(q)$.

The projective symplectic group is $\PSp_{2m}(q) = \Sp_{2m}(q)/Z$, where $Z$ is the center of $\Sp_{2m}(q)$.

The groups $\Sp_2(q) \cong \SL_2(q)$ and $\PSp_2(q)\cong \PSL_2(q)$ have been already considered in the previous section. For $m \geq 2$, the groups $\PSp_{2m}(q)$ are simple, except $\PSp_4(2) \cong \Sp_4(2) \cong S_6$. 

If $p \neq 2$, then the principal $p$-blocks of $\Sp_n(q)$ and $\PSp_n(q)$ are the same, because $|Z|=(2,q-1)$. Moreover, $\Sp_{2m}(q) \cong \PSp_{2m}(q)$ for even $q$. 

\begin{proposition}\label{symp}
Let $G$ be one of the groups $\PSp_{2m}(q)$ or $\Sp_{2m}(q)$, where $m \geq 2$, and let $p$ be a prime dividing $|G|$. Then $G \notin \mathfrak{X}_p$.
\end{proposition}
\begin{proof}
Using Lemmas \ref{2-sylow} and \ref{p_div_q}, we may assume that $p > 2$ and $p \nmid q$. In this case, Sylow subgroups of $\PSp_{2m}(q)$ and $\Sp_{2m}(q)$ are isomorphic, and the principal blocks of this groups are the same, because $p$ doesn't divide $|Z(\Sp_{2m}(q))|$. Thus, it suffices to consider only $\Sp_{2m}(q)$.

Assume that $G = \Sp_{2m}(q) \in \mathfrak{X}_p$ for some $m, q$ (where $m \geq 2$). 

Let $d = \ord_p(q)$. If $d=1$, then, by Lemma \ref{GL_sylow_cyclic}, a Sylow $p$-subgroup of $\GL_m(q)$ is not cyclic for $m\geq 2$. The diagonal embedding \ref{emb_sp} shows that the same holds true for a Sylow subgroup $P$ of $\Sp_{2m}(q)$. 

Therefore, we can assume that $d \neq 1$. By Lemma \ref{e}, the number $e$ of edges in the Brauer tree $\tau_0(G)$ is equal to $|N_G(P)/C_G(P)|$.

We will show that $e \geq d \geq 3$, that gives a contradiction with the assumption, because $\tau_0(G)$ is a line by Lemma \ref{polyg}.

The structure of $P$ depends on $d$. 

(1) \emph{Even $d$}. According to \cite[Table 1]{Stat}, $P$ coincides with a Sylow $p$-subgroup of the ambient group $\GL_{2m}(q)$. Since $P$ is non-trivial and cyclic, we conclude that $m< d\leq 2m$. Further, $m\geq 2$ yields $d\geq 4$.

Under a suitable choice of the matrix $W$, the group $\Sp_d(q)\times \Sp_{2m-d}(q)$ is embedded into $\Sp_{2m}(q)$ as a block diagonal (see \cite[Remark 3.2]{Alali} for details); hence, $P$ can be chosen in the left upper corner $G_d = \Sp_d(q)$. 
Now taking into account Lemma \ref{e_subgroup}, it suffices to show that $e_d= |N_{G_d}(P)/C_{G_d}(P)| \geq d$.

Let $\al$ be a generator of $P$. According to \cite[Lemma 4.6]{Stat}, $N_{G_d}(P)/C_{G_d} (P)$ is a cyclic group generated by an element $y$ acting by conjugation as $\al^y= \al^d$. Since this action has order $d$, we obtain $e_d \geq d$. Thus, $e \geq e_d \geq 4$ as desired.

(2) \emph{Odd $d$}. According to \cite{Stat}, the order of $P$ is equal to the order of a Sylow $p$-subgroup $P'$ of $\GL_m(q)$. We may assume that $P$ is the image of $P'$ under embedding \ref{emb_sp}.
Since $P'$ is cyclic, we have $m/2< d\leq m$. Thus, $d\geq 3$.

It remains to show that $e\geq d$. As above, we may assume that $m=d$. Consider an element $y'\in \GL_d(q)$ which acts by conjugation on $P'$ as an automorphism of order $d$. The diagonal image of this element belongs to $N_G(P)$ and acts as an automorphism of order $d$ on this subgroup.
\end{proof}

\vspace{10mm}
\subsection{Unitary groups}\label{S-unitary}

Let us denote by $\ov{\phantom{r}}$ the involution $a \mapsto a^q$ of the field $\F_{q^2}$, and let $V$ be an $n$-dimensional vector space over this field. Then there exists an unique non-singular conjugate-symmetric sesquilinear form $f: V\times V\to \F_{q^2}$, i.e. $f(u,v)= \ov{f(v,u)}$ for all $u, v\in V$.
If $f$ is given by a matrix $W$, then $W= \ov W^t$. For instance, $W$ can be the identity matrix.

The general unitary group $\GU_n(q^2)$ consists of matrices $A\in \GL_n(q^2)$ preserving $f$, i.e. $A W\ov A^{\, t}= W$. The order is
$$
|\GU_n(q^2)|= q^{n(n-1)/2}\cdot (q+1)\cdot (q^2-1)\cdots\ldots\cdot (q^n- (-1)^n)\,.
$$

The unitary matrices of determinant 1 form a normal subgroup in $\GU_n(q^2)$ of index $q+1$, called the special unitary group $\SU_n(q^2)$. The center $Z$ of $\SU_n(q^2)$ consists of scalar matrices, and $|Z| = (n, q+1)$. The quotient group $\PSU_n(q^2)= \SU_n(q^2)/Z$ is called the projective special unitary group. If $n\geq 3$, then this group is simple, except $\PSU_3(2^2)$.

Note that $\PSU_2(q^2)\cong \PSL_2(q)$. Thus, it suffices to consider $n\geq 3$. 

\begin{proposition}\label{P-unitary}

Let $G$ be one of the groups $\PSU_{n}(q^2)$ or $\SU_{n}(q^2)$, where $n \geq 3$, and let $p$ be a prime dividing $|G|$. Then $G \in \mathfrak{X}_p$ if and only if $n=3$, $p>2$ and $p$ divides $q-1$. 

\end{proposition}
\begin{proof}
As before, we assume that $p \neq 2$ and $p \nmid q$. Write $d = \ord_p(q)$, and $s = \ord_p(q^2)$.

We first consider the case $n=3$. If $p=3$ and $p\mid q+1$, then Sylow subgroups of $\PSU_3(q^2)$ and $\SU_3(q^2)$ are not cyclic. Overwise, the Sylow subgroups are cyclic, and the principal $p$-block of $\PSU_3(q^2)$ coincides with the principal $p$-block $B_0$ of $ \SU_3 (q^2)$, because $B_0$ is annihilated by the center of the group $\SU_3 (q^2)$. It follows from  \cite{Geck90} that the Brauer tree of $B_0$ is a start if and only if $p \neq 2$ and $p \mid q-1$.

Now set $n\geq 4$. In this case, if $p$ divides $q\pm 1$, Sylow $p$-subgroups of $\PSU_n(q^2)$ and $\SU_n(q^2)$ are not cyclic. 

Assume that $p \nmid q\pm 1$. Then $2< d\leq 2n$. Further, the principal block of $\PSU_n (q^2)$ coincides with the principal block $B_0$ of the group $\SU_n(q^2)$. 

A Sylow $p$-subgroup $P$ of $H=\SU_n(q^2)$ coincides with a Sylow $p$-subgroup of $G=\GU_n(q^2)$. Moreover, $|N_G(P)/C_G(P)| = |N_H(P)/C_H(P)|$. Therefore, $\tau_0(H) = \tau_0(G)$. This tree is a line by Lemma \ref{polyg} (for odd and even $q$). Thus, it suffices to show that the number of edges in $\tau_0(G)$ is greater than $2$.

There are two possibilities for $P$ in depending on $d$.

(1) \emph{$d\equiv 2 \pmod 4$}. According to \cite{Stat}, $P$ is also a Sylow $p$-subgroup of the ambient group $\GL_n(q^2)$. 
Because $P$ is cyclic, we conclude that $n/2< s = d/2 \leq n$. In particular, $s \geq 3$. 

The group $\GU_s(q^2)\times \GU_{n-s}(q^2)$ is embedded into $\GU_n(q^2)$, and $P$ can be chosen in the left upper corner $G_s = \GU_s(q^2)$. As well as for symplectic groups, it is easy to show that the number of edges
$e_0(G) \geq e_0(G_s) = |N_{G_s}(P)/C_{G_s}(P)| \geq s$, as desired.

(2) \emph{$d\equiv 0, 1, 3 \pmod 4$}. Let $n= 2m+ \eps$, where $\eps= 0, 1$. According to \cite{Stat}, the order of $P$ equals the order of a Sylow $p$-subgroup $P'$ of $\GL_m(q^2)$. If $\eps=0$, then choosing $W= \bsm 0&I_m\\I_m&0\esm$, we obtain the embedding from $\GL_m(q^2)$ into $\GU_n(q^2)$, which sends $A$ to $\bsm A&0\\0& \ov{A}^{\,-t}\esm$. And a similar embedding takes place for $\eps=1$ if we add $1$ to the lower right corner of $W$. 

We first consider the case when $d$ is odd; hence, $s= d> 2$. Because $P'$ is cyclic and non-trivial, we conclude that $m/2< s\leq m$. As above, to estimate $e_0(G)$, we may assume that $n= 2d$. Choose an element in $\GL_d(q^2)$ which acts by conjugation on $P'$ as an automorphism of order $s$. By multiplying by a constant we may assume that this element has determinant $1$. Expanding diagonally, we conclude that $e\geq s> 2$, as desired.

It remains to consider the last case when $d$ is divisible by $4$. In this case, $d= 2s$. If $d> 4$, then, using the diagonal embedding, we conclude that $e\geq s > 2$. 

Finally, if $d=4$, i.e. $G = \GU_4(q^2)$ and $p\mid q^2+1$, then the generator $A=\al'$ for $P'$ can be chosen such that $A\cdot \ov A^{\, t}= I_2$. As above, using the diagonal embedding we obtain that $e\geq 2$. But also the matrix $W$ normalizers $P$, hence $e\geq 4$, as desired.
\end{proof}

\vspace{10mm}
\subsection{Orthogonal groups in odd dimension}\label{S-ort-odd}

Let $n= 2m+1$, where $m\geq 1$, and let $V$ be an $n$-dimensional vector space over  $\mathbb{F}_q$. The general orthogonal group $\GO_n(q)$ consists of invertible matrices of order $n$ which preserve a scalar product given by a matrix $W$, i.e. $A\in \GO_n(q)$ if and only if $AWA^t = W$. For instance, we may take \rule{0pt}{6mm}$W= \bsm 0&I_m&0\\ I_m&0&0\\0&0&1\esm$. It gives the embedding 
\begin{equation}\label{eq-GL-GO}
\GL_m(q) \to \GO_{2m+1}(q), ~ A \mapsto \bsm A&0&0\\0&A^{-t}&0\\0&0&1\esm. 
\end{equation}

\vspace{2mm}

The group $\GO_n(q)$ contains a special orthogonal group $\SO_n(q)$, consisting of matrices with determinant $1$. The derived subgroup of $\SO_n(q)$ is denoted by $\Om_n(q)$. 

If $q$ is even, then $\Om_{2m+1}(q) \cong \PSp_{2m}(q)$ for any $m \geq 1$. If $q$ is odd, then $\Om_3(q)\cong \PSL_2(q)$ and $\Om_5(q)\cong \PSp_4(q)$, according to \cite[p. 43]{K-L}.

Thus, it suffices to consider the groups $\Om_{2m+1}(q)$ for odd $q$ and $m \geq 3$. They are all simple, and the order is
$$
|\Om_{2m+1}(q)|= \frac{\scriptstyle 1}{\scriptstyle 2}\cdot q^{m^2}\cdot (q^2-1)\cdot (q^4-1)\cdot\ldots \cdot (q^{2m}-1)\,.
$$

\begin{lemma}\label{GO-SO}
Suppose that $G = \GO_{2m+1}(q)$ and $H = \SO_{2m+1}(q)$, where $q$ is odd. Suppose that Sylow $p$-subgroup $P$ of $H$ is cyclic and non-trivial, and that $p$ does not divided $q$. Then the Brauer trees of the principal blocks of $G$ and $H$ are the same.
\end{lemma}
\begin{proof}
Since $\tau_0(H)$ is similar to $\tau_0(G)$ by Lemma \ref{L-tree}, it is sufficient to show that $|N_G(P)/C_G(P)| = |N_H(P)/C_H(P)|$.

The equality $|G| = |H|/2$ yields $|N_H(P)| = |N_G(P)|/2$. To deduce the same equality for centralizers of $P$, it suffices to show that there exists a matrix $Z \in G$ such that $\det(Z) = -1$ and $Z$ centralizes $P$. One may take $Z = \mathrm{diag}(1,1,...,1,-1)$.
\end{proof}

\begin{theorem}\label{T-ort}
Let $G$ be one of the groups $\Om_{2m+1}(q)$, $\SO_{2m+1}(q)$ or $\GO_{2m+1}(q)$, where $m\geq 3$ and $q$ is odd. Let $p$ be a prime dividing the order of $G$. Then $G \notin \mathfrak{X}_p$.
\end{theorem}
\begin{proof}
Assume that $G \in \mathfrak{X}_p$. Then a Sylow $p$-subgroup $P$ of $G$ is cyclic, and $p \neq 2$, $p \nmid q$. 

Since all indices in the normal series $\Om_n(q) \subset \SO_n(q) \subset \GO_n(q)$ are equal to $2$, the Brauer trees of the principal blocks of these groups are similar to each other. Moreover, by Lemma \ref{GO-SO}, the principal blocks of $\SO_n(q)$ and $\GO_n(q)$ have the same Brauer trees. Further, since this tree is a line (by Lemma \ref{polyg}), it suffices to show that the number $e$ of edges in  $\tau_0(\GO_n(q))$ is larger than $4$, or equals $4$ but the exception vertex is not in the center.

Write $d = \ord_p(q)$.

(1) \emph{Even $d$}. In this case, $P$ can be chosen as a subgroup of the ambient group $\GL_{2m}(q)$. Because $P$ is cyclic and non-trivial, we obtain $m< d\leq 2m$; hence, $d> 3$. Using \cite[Lemma 4.5]{Stat}, we derive that $e\geq d$. Thus, if $m\geq 4$, then $d\geq 6$ implies $e\geq 6$, a contradiction.

Consider the remaining case $m=3$ and $d=4$, i.e. $G = \SO_7(q)$, where $p \mid q^2+1$. 
Using \cite[p. 466--467]{Car}, we can calculate the Brauer tree of the principal block. It has the following shape, where the numbers near vertices show the degrees and parameterization symbols of characters.
$$
\vcenter{%
\def\labelstyle{\displaystyle}
\xymatrix@C=68pt@R=10pt{%
*+={\circ}\ar@{-}[r]\ar@{}+<0pt,-14pt>*{_{\bsm 3\\\varnothing\esm}}\ar@{}+<0pt,10pt>*{_1}&
*+={\circ}\ar@{-}[r]\ar@{}+<0pt,-10pt>*{_{\bsm 0&2\\2&\esm}}\ar@{}+<-3pt,10pt>*{_{q^2(q^4+q^2+1)}}&
*+={\circ}\ar@{-}[r]\ar@{}+<0pt,-10pt>*{_{\bsm 0&1&2\\1&3\esm}}\ar@{}+<0pt,10pt>*{_{q^4(q^3+1)(q+1)/2}}&
*+={\bullet}\ar@{-}[r]\ar@{}+<0pt,-10pt>*{_{}}\ar@{}+<3pt,10pt>*{_{(q^6-1)(q^2-1)}}&
*+={\circ}\ar@{}+<0pt,-10pt>*{_{\bsm &1\\0&1&2&3\esm}}\ar@{}+<4pt,10pt>*{_{q^4(q^3-1)(q-1)/2}}\\
}}
$$

\vspace{3mm}

The character degrees are comparable with $1, -1, 1, 4, 1$ modulo $p$; therefore, the second on the right character is exceptional (see \cite[Theorem 7.2.16]{Feit}).

(2) \emph{Odd $d$}. In this case, the order of $P$ equals the order of a Sylow $p$-subgroup $P'$ of $\GL_m(q)$. Thus, $P$ is the image of $P'$ via the embedding \eqref{eq-GL-GO}.

Since $P$ is cyclic, we conclude that $m/2< d\leq m$; hence, $d\geq 3$. Calculating the normalizer of $P'$ in $\GL_d(q)$, we get $e\geq d$. 
If $m > 5$, then $d > 3$ that gives the desired. 

Now consider the remaining cases with $3 \leq m \leq 5$ and $d= 3$ (i.e. $p\mid q^2+q+1$). It follows from Lemma \ref{independ_q} that the number of edges $e_0(G)$ in the principal block of $G = \SO_{2m+1}(q)$ depends on $d$ only. For instance, we may take $q=3$ and $p=13$. A calculation in GAP \cite{gap} gives $e_0(G) =|N_G(P)/C_G(P)| = 6$ for $m= 3, 4, 5$.

\end{proof}

\subsection{Orthogonal groups in even dimension.}\label{S-ort-even}

Write $n=2m$, where $m\geq 1$. A general orthogonal group $\mathrm{GO}_n^\pm(q)$ consists of invertible matrices of order $n$, preserving the quadratic form correspondingly $Q^{(+)} = \bsm 0 & I_m \\ 0 & 0 \esm$ and \rule{0pt}{7mm}$Q^{(-)} = \bsm 0&I_{m-1}&0&0\\0&0&0&0\\0&0&1&a\\0&0&0&b\esm$, where $a= \ga+ \ga^q$, $b= \ga^{q+1}$, and $\ga$ is a primitive element of $\F_{q^2}$.
The corresponding bilinear form for $Q^{(\pm)}$ is $W = Q + Q^T$.

The conformal orthogonal group $\CO^\pm_n(q)$ consists of invertible matrices which preserve this form up to a multiplicative constant. The special orthogonal group $\SO^+_n(q)$ is a subgroup of $\GO_n^\pm(q)$ of matrices with determinant 1. The conformal special orthogonal group $\CSO^\pm_{2m}(q)$ consists of matrices $A$ such that the multiplicative constant $\lam= \lam(A)$ satisfies $\lam^m= \det A$ (see \cite[p. 13]{Liv}). 

For odd $q$, we have the following diagram of normal subgroups of $\CO^\pm_n(q)$.

$$
\vcenter{%
\xymatrix@C=20pt@R=17pt{%
&*+={\circ}\ar@{-}[dl]_2\ar@{-}[dr]^{q-1}\ar@{}+<0pt,10pt>*{_{\CO^\pm_n(q)}}&\\
*+={\circ}\ar@{-}[dr]_{q-1}\ar@{}+<-22pt,0pt>*{_{\CSO^\pm_n(q)}}&&
*+={\circ}\ar@{-}[dl]_2\ar@{}+<20pt,0pt>*{_{\GO^\pm_n(q)}}\\
&*+={\circ}\ar@{-}[d]_2\ar@{}+<20pt,-3pt>*{_{\SO^\pm_n(q)}}&\\
&*+={\circ}\ar@{}+<18pt,-4pt>*{_{\OM^\pm_n(q)}}&
}}
$$

\vspace{3mm}

For even $q$, we have $\GO^\pm_n(q) \cong \SO^\pm_n(q)$ and $\CO^\pm_n(q) \cong \CSO^\pm_n(q)$.

The centre $Z$ of $\OM^\pm_n(q)$ consists of scalar matrices, and $|Z| = (4,q^m-1)/2$. The factor group $\POM^\pm_n(q) = \OM^\pm(q)/Z$ is called a projective orthogonal group. It has the following order.
$$
|\POM^\pm_{2m}(q)|= \frac{\scriptstyle 1}{\scriptstyle (4, q^m-1)}\cdot q^{m(m-1)/2}\cdot (q^m \mp 1)\cdot \prod_{i=1}^{m-1} (q^{2i}-1)\,.
$$

Note (see \cite[p. 43]{K-L}) that the group $\OM^+_2(q)$ is isomorphic to the dihedral group $D_{2(q-1)}$, which is solvable.

Further, the group $\POM^+_4(q)\cong 2. (\PSL_2(q)\times \PSL_2(q))$ is not simple. The simple groups $\POM_4^-(q)\cong \PSL_2(q^2)$ have been already considered.

If $m\geq 3$, all groups $\POM^\pm_{2m}(q)$ are simple. We already know the answer for $\POM^+_6(q) \cong \PSL_4(q)$ and $\POM_6^-(q)\cong \PSU_4(q^2)$. Thus, it is remaining to consider the groups with $m \geq 4$.

\begin{theorem}\label{T-ort+-}
Let $G$ be one of groups $\POM^\pm_{2m}(q)$, $\OM^\pm_{2m}(q)$, $\SO^\pm_{2m}(q)$, $\GO^\pm_{2m}(q)$, where $m\geq 4$, and let $p$ be a prime dividing the order of $G$. Then $G \notin \mathfrak{X}_p$. 
\end{theorem}
\begin{proof}
As usual, we exclude the case $p=2$, and also when $p$ divides $q$ or $q-1$. Thus, we may assume that $d\geq 2$, and a Sylow $p$-subgroup $P$ of $G$ is cyclic and non-trivial. 

First, assume that $q$ is odd. By Fact \ref{polyg}, $\tau_0(\CSO^\pm_n(q))$ is a line. By Lemma \ref{L-tree}, the Brauer trees of the principal blocks of $\SO^\pm_n(q)$ and $\OM^\pm_n(q)$ are similar to $\tau_0(\CSO^\pm_n(q))$, hence they are lines, too. Moreover, $\tau_0(\GO^\pm_n(q)) = \tau_0(\SO^\pm_n(q))$, because they have the same number of edges. 
Also $\tau_0(\POM^\pm_n(q))$ coincide with  $\tau_0(\OM^\pm_n(q))$ because $Z \subseteq O_{p'}(\OM^\pm_n(q))$.
Therefore, $e_0(\GO^\pm_n(q)) > 4$ implies $e_0(\POM^\pm_n(q)) > 2$.

Now assume that $q$ is even. Then $\POM^\pm_n(q) \cong \OM^\pm_n(q)$. By Fact \ref{polyg}, the Brauer tree of the principal block of $\GO^\pm_n(q) \cong \SO^\pm_n(q)$ is a line. And, by Fact \ref{L-tree}, the tree $\tau_0(\OM^\pm_n(q))$ is similar to $\tau_0(\GO^\pm_n(q))$.

Thus, for both odd and even $q$, it suffices to show that the number of edges in $\tau_0(\GO^\pm_n(q))$ exceeds $4$.

We first consider the group $G = \GO^+_n(q)$. Since a Sylow $p$-subgroup $P$ of $G$ is cyclic, it gives that $d\leq 2m-2$. According to \cite[Table 1]{Stat}, we have the following possibilities for $P$.

(1.1) \emph{Odd $d$}. Then $P$ can be chosen in $\GL_m(q)$. Since $P$ is non-trivial and cyclic, we obtain $m/2< d\leq m$. Using the diagonal embedding $A\mapsto \bsm A&0\\ 0&A^{-t}\esm$ from $\GL_m(q)$ into $\GO^+_{2m}(q)$, we obtain $e\geq d$. 

If $m\geq 8$, then $e > 4$, as desired. Since $e\geq d$, it remains to consider only cases when $d=3$ and $m \in \{4,5\}$. By Lemma \ref{independ_q}, the shape of $\tau_0(\GO^+_m(q))$ depends only on $d$ and $m$.
Taking $p=7$ and $q=2$, we have found with GAP that $e = 6$ for both $m \in \{4,5\}$.

(1.2) \emph{Even $d$ is even}. The inequality $d\leq 2m-2$ implies that the integer part of the fraction $d/2m$ equals zero; hence, $P$ can be chosen as a subgroup of $\GL_{2m}(q)$. Since $P$ is cyclic and non-trivial, we obtain $m< d$. Now, using \cite[Lemma 4.6]{Stat}, we get $e\geq d$. Then $d> 4$ yields $e> 4$, as desired.

Now consider $G = \GO^-_n(q)$. Since a Sylow $p$-subgroup $P$ of $G$ is cyclic, we have $2\leq d\leq 2m$. Notice that $d=2m$ may occur, when $p \mid q^m+1$. We will consider possibilities for $P$, according to \cite[Table 1]{Stat}.

(2.1) \emph{Odd $d$}. Then $P$ can be chosen as a subgroup of $\GL_{m-1}(q)$. Since $P$ is cyclic, we conclude that $(m-1)/2< d$. By \cite[Lemma 4.6]{Stat}, we have $e\geq d$.

If $m \geq 7$, then $e\geq d\geq 5$, as desired. We should consider only $m \in \{4,5,6\}$ and $d=3$. In all these cases, we have found with GAP that $e=6$. 

(2.2) \emph{Even $d$}. This case splits into two subcases.

If the integer part of $d/2m$ is odd, then $d\leq 2m$ yields $d=2m$. From \cite[Table 1]{Stat} we see, that $P$ can be chosen as a subgroup of $\GL_{2m}(q)$. Since $P$ is cyclic, we conclude that $m< d$ yields $d\geq 5$. Using \cite[Lemma 4.6]{Stat}, we conclude that $e\geq d\geq 5$, as desired.

Otherwise, the integer part of $d/2m$ is zero. Hence, $P$ can be chosen as a subgroup of $\GL_{2m-2}(q)$. Since $P$ is cyclic, we get $m-1< d$. If $m > 4$, it follows from \cite[Lemma 4.6]{Stat} that $e\geq d > 4$.

For the remaining case when $m=d=4$, we have found with GAP that the number of edges in the principal blocks of both groups $\GO^-_{8}(q)$ and  $\POM^-_{8}(q)$ are equal $4$. This completes the proof.
\end{proof}

\section{Exceptional groups of Lie type} 

In this section, we consider the finite exceptional groups of Lie type, namely  $E_6$, $E_7$, $E_8$, $F_4$, $G_2$ and twisted groups $^2 B_2$, $^3 D_4$, $^2E_6$, $^2 F_4$, $^2 G_2$. The last two types are called Ree groups. The groups of the type ${}^2 B_2$ are known as Suzuki groups. 

All these groups are simple except $G_2(2)$ and $^2 F_4(2)$. The group $G_2(2)$ has order $2^6 \cdot 3^3 \cdot 7$, and its derived subgroup $G_2(2)' \cong \mbox{PSU}(3,3)$ is simple of order $2^5 \cdot 3^3 \cdot 7$. 

The group $^2F_4(2)$ has the derived subgroup $^2F_4(2)'$ which is simple and called the Tits group. 

\begin{proposition}
Let $G = {}^2 F_4(2)'$ and $p$ divides $|G|$. Then $G \notin \mathfrak{X}_p$.
\end{proposition}
\begin{proof}
The order of $G$ is $2^{11}\cdot 3^3\cdot 5^2\cdot 13$. Since $G$ contains maximal subgroups $\PSL(3,3).C_2$ and $\PSL(2,25)$ (see \cite[p. 74]{ATLAS}), we conclude that a Sylow $p$-subgroup of $G$ is not cyclic for $p= 2, 3, 5$.

If $p= 13$, then it follows from decomposition matrices \cite{Bre} that $\tau_0(G)$ is not a star. 
\end{proof}

\begin{proposition}
Let $G$ be a simple group of any type $^3 D_4$, $E_6$, $^2E_6$, $E_7$, $E_8$, $F_4$, $^2 F_4$ or $G_2$. Then $G \notin \mathfrak{X}_p$ for any $p$ dividing the order of $G$.
\end{proposition}

\begin{proof}
Assume that $G \in \mathfrak{X}_p$ for some $p$ dividing $|G|$. Then Sylow $p$-subgroup $P$ of $G$ is cyclic and $p\neq 2$, $p \nmid q$. 
Let $d = \ord_p(q)$. In particular, $d$ divides $p-1$. We will consider groups from the list one by one.

\begin{enumerate}

\item Let $G= {}^3D_4(q)$, where $q$ is a prime power. It has the order
$$
|^3D_4(q)|= q^{12}(q^2-1)^2 (q^4- q^2+1) (q^4+ q^2+1)^2\,.
$$

Since $P$ is cyclic, it follows from \cite[Prop.~5.6c and Table~1.1]{D-M} that $P$ is a subgroup of the maximal torus $T_5= C_{q^4- q^2+1}$, i.e. $p$ divides $q^4- q^2+ 1$. In this case according to \cite[p.~3265]{Geck91}, the Brauer tree of $B_0$ is a line with 4 edges; hence, $\tau_0(G)$ is not a star, a contradiction.

\item Let $G= E_6(q)$. Looking at the order
$$
|E_6(q)|= \frac 1{(3, q-1)}\cdot q^{36} (q^2-1)(q^5-1)(q^6-1)(q^8-1)(q^9-1)(q^{12}-1)\,, 
$$

we see that the possible values for $d$ are $1, \dots, 6, 8, 9, 12$.

If $p=3$, then $p$ divides either $q+1$ or $q-1$. According to \cite[p. 897]{D-F}, $G$ contains the maximal torus $C_{q+1}^2\times C_{q^2-1}^2$, so Sylow 3-subgroups of $G$ are not cyclic.

Assume now that $p > 3$. If $d \in \{ 1, 6 \}$, then by \cite[Theorem 3.1]{HLM} there are no unipotent blocks with cyclic defect group. Note that the principal $p$-block is always unipotent, and its cyclic defect group is a Sylow $p$-subgroup of $G$.

For $d \in \{ 3,4,5,8,9,12 \}$, according to \cite[Theorem 3.1]{HLM}, there are no unipotent blocks whose Brauer tree is a star. 

If $d = 2$, then $p \mid q+1$. In this case, $P$ is not cyclic because $G$ contains a subgroup which is isomorphic to $C_{q+1} \times C_{q+1}$.

\item  Let $G= {}^2E_6(q)$. Since
$$
|^2E_6(q)| = \frac 1{(3, q+1)}\cdot q^{36} (q^2-1)(q^5+1)(q^6-1)(q^8-1)(q^9+1)(q^{12}-1)\,,
$$
the possible values for $d$ are $1, 2, 3, 4, 6, 8, 10, 12, 18$.

Again we first consider the case $p> 3$. If $d \in \{ 2, 3 \}$, by \cite[Theorem~2.2]{H-L_F4} the group $G$ has no unipotent blocks with cyclic defect group. If $d \in \{ 4, 6, 8, 10, 12, 18 \}$, then there are no unipotent blocks whose Brauer tree is a star and whose defect group is cyclic.

Finally, assume that $d=1$, i.e. $p \mid q-1$. According to \cite[p.~903]{D-F}, the maximal tori of $G$ can be obtained from the corresponding list for $E_6(q)$ by a formal substitution $q\mapsto -q$. 
It follows from \cite[p.~897]{D-F} that $G$ contains the maximal torus $T_{11}= C_{q^2-1}\times C_{q^4-1}$. Therefore, $P$ is not cyclic.

Thus, it remains to consider the case $p=3$. Since $q\equiv \pm 1 \pmod 3$, we conclude that $P$ is not cyclic by considering the same torus as above.

\item Let $G= E_7$. Since
$$
|E_7(q)| = \frac 1{(2, q-1)}\cdot q^{63} (q^2-1)(q^6-1)(q^8-1)(q^{10}-1)(q^{12}-1)(q^{14}-1)(q^{18}-1)\,,
$$
the possible values for $d$ are $1, \dots, 10, 12, 14, 18$.

If $d \in \{ 5, 7, 8, 9, 10, 12, 14, 18 \}$, then it follows from \cite[Theorem~12.6, and remark on p.~2970]{Geck92} that all unipotent blocks with cyclic defect group have Brauer tree in the shape of a line of length $e \geq 4$.

For remaining values of $d$, considering the maximal tori of $G$, it is easy to see that a Sylow $p$-subgroup of $G$ is not cyclic. 
Indeed, we can take the torus $T_{10}= C_{q-1}\times C_{q^3-1}^2$ (see \cite[p.~898]{D-F}) for $d \in \{ 1, 3 \}$, 
the torus $T_8= C_{q-1}\times C_{q+1}^2\times C_{q^2-1}^2$ for $d= 2$, 
and $T_{28}= C_{(q-1)(q^2+1)}\times C_{q^2-1}\times C_{q^2+1}$ for $d=4$. 
Finally, for $d=6$ replacing $q$ by $-q$ in $T_{10}$, we see that the group contains the torus $C_{q+1}\times C_{q^3+1}^2$.

Note that the case $p=3$, when either $p\mid q+1$ or $p\mid q-1$, is already included into consideration above.

\item Let $G= E_8$. The order is
$$
|E_8(q)| = q^{120} (q^2-1)(q^8-1)(q^{12}-1)(q^{14}-1)(q^{18}-1)(q^{20}-1)(q^{24}-1)(q^{30}-1)\,.
$$

Hence the possible values for $d$ are $1, \dots, 10, 12, 14, 15, 18, 20, 24, 30$.

If $d \in \{ 7, 9, 14, 15, 18, 20, 24, 30 \}$, then it follows from \cite[Thm. 12.7]{Geck92} that the Brauer trees of all unipotent blocks with cyclic defect group are not stars.

For remaining values of $d$, a Sylow $p$-subgroup of $G$ is not cyclic because $G$ contains the following tori (see \cite[p.~899--901]{D-F}): $T_{34}= C_{(q+1)(q^3-1)}^2$ for $d = 1, 2, 3$; $T_{36}= C_{q^4-1}^2$ for  $d = 4$; $T_{57}= C_{q^4+q^3+q^2+q+1}^2$ for $d = 5$; $T_{61}= C_{q^4+1}^2$ for $d = 8$; and $T_{67}= C_{q^4- q^2+1}^2$ for $d = 12$. For $d = 6$, we can take the torus $C_{(q^4-q^2+1)(q^2-q+1)}\times C_{q^2-q+1}$ that obtained from $T_{62}$ by plugging $q \mapsto -q$. And for $d = 10$, we obtain the torus $C_{q^4-q^3+q^2-q+1}$ from $T_{57}$ by the same substitution.

\item Let $G= F_4(q)$. Then 
$$
|F_4(q)| = q^{24} (q-1)^4 (q+1)^4 (q^2-q+1) (q^2+1)^2 (q^2+q+1) (q^4- q^2+ 1) (q^4+1) (q^4+ q^2+ 1)\,,
$$

The possible values for $d$ are $1, 2, 3, 4, 6, 8, 12$.

First assume that $p > 3$. If $d \in \{ 1, 2, 3, 6 \}$, then it follows from \cite[Theorem~2.1]{H-L_F4} that there are no unipotent blocks with cyclic defect group. 
If $d \in \{ 4, 8, 12 \}$, then by the same reference the Brauer tree of unipotent blocks is not a star.

Now suppose that $p=3$. It follows from \cite[Table~4.7.3]{GLS} that $G$ has three conjugacy classes of $3$-elements. Hence, Sylow $3$-subgroups of $G$ are not cyclic.

\item Let $G= {}^2F_4(q^2)$, where $q^2= 2^{2n+1}$, $n > 0$. Then $G$ is a simple group of the order
$$
q^{24} (q^2-1)^2 (q^2+1)^2 (q^4- q^2+1) (q^4+1)^2 (q^4\pm \sqrt{2} q^3+ q^2\pm \sqrt{2} q+1)\,.
$$

If $P$ is cyclic, then it follows from \cite[Section~4.2]{His91} that $p> 3$ and $p$ divides
$q^4- q^2+ 1$ or $q^4\pm \sqrt{2} q^3+ q^2\pm \sqrt{2} q+1$. According to \cite[Theorems 4.5--4.7]{His91}, in each of these cases $\tau_0(G)$ is not a star.

\item Let $G= G_2(q)$, where $q \geq 3$. This group is simple, and its order is 
$$
|G_2(q)|= q^6 (q^2-1)^2(q^2-q+1)(q^2+q+1)\,.
$$

If $p$ divides $q\pm 1$, then the Sylow $p$-subgroup of $G$ is not cyclic, because $G$ contains the tori $C_{q\pm 1}^2$ (see \cite[p.~1902]{Sham89a}).

Suppose that $p$ divides $q^2\pm q+1$. If $p\neq 3$, then by \cite[p.~380--381]{Sham89b} $\tau_0(G)$ is not a star.

Now assume that $p=3$. If $3$ divides $q^2+q+1$, then $q\equiv 1\pmod 3$, and this case has been already considered above. Similarly, $3\mid q^2-q+1$ leads to $3\mid q+1$.

\end{enumerate}
\end{proof}

Now consider the series of groups ${}^2 G_2(q^2)$, where $q^2 = 3^{2n+1}$. They are simple if and only if $n \geq 1$. The order is
$$|{}^2 G_2(q^2)| = q^6(q^2-1)(q^6+1) = q^6(q^2-1)(q^2+1)(q^2+ \sqrt{3} q+1)(q^2-\sqrt{3} q+1)\,.$$

\begin{proposition}
Let $G= {}^2 G_2(q^2)$, where $q^2 = 3^{2n+1}$, $n \geq 1$. Then $G \in \mathfrak{X}_p$ if and only if $p>2$ and $p$ divides either $q^2-1$ or $q^2+ \sqrt{3} q+1$.
\end{proposition}
\begin{proof}
Again we may assume that $p \nmid q$ and $p > 2$. According to \cite[Theorems 4.1-4.4]{His91}, if $p$ divides $q^2+1$ or $q^2- \sqrt{3} q+1$, then $P$ is cyclic and $\tau_0(G)$ is a star. 
If $p$ divides $q^2+ \sqrt{3} q+1$ or $q^2-1$, then $P$ is cyclic, but $\tau_0(G)$ is not a star.

\end{proof}

Finally, we will consider the Suzuki groups ${}^2 B_2(q)$.

\begin{proposition}
Let $G = {}^2 B_2(q)$, where $q=2^{2n+1}$, $n \geq 1$, and let $r = 2^{n+1}$. Then $G \in \mathfrak{X}_p$ if and only if $p$ divides either $q-1$ or $q+r+1$.

\end{proposition}
\begin{proof}
Since $r^2 = 2q$, the order of $G$ can be written as 
$$|G| = q^2(q^2+1)(q-1) = q^2(q+r+1)(q-r+1)(q-1) \,.$$

It is known that $|G|$ is not divisible by $3$, and that a Sylow $p$-subgroup of $G$ is cyclic for $p>3$. The tree $\tau_0(G)$ is described in \cite{Burk79}. It is a star when $p$ divides either $q-1$ or $q+r+1$, and it is not a star when $p$ divides $q-r+1$.
\end{proof}

 
\vspace{2mm}
\emph{The work was supported by Russian Science Foundation (grant 18-71-10007).}

\end{document}